\documentclass[dvipdfmx,autodetect-engine]{journal}
\usepackage{tikz}
\usepackage{color}
\usepackage[all]{xy}
\usepackage{amsmath, amssymb}
\usepackage{amsfonts}
\usepackage{mathrsfs}
\usepackage{amsthm}
\usepackage{bm}

\makeatletter
 \def\th@plain{\upshape}
 \makeatother
\newtheorem{theorem}{Theorem}[section]
\newtheorem{prop}[theorem]{Proposition}
\newtheorem{lemma}[theorem]{Lemma}
\newtheorem{cor}[theorem]{Corollary}
\newtheorem{Example}[theorem]{Example}
\newtheorem{definition}[theorem]{Definition}
\newtheorem{rem}[theorem]{Remark}

\title{Note on spherical quandles}
\author{
Kentaro Yonemura
\thanks{
the 2010 MSC: \texttt{57K10, 57K12}, E-mail: \texttt{3MA20009Y@s.kyushu-u.ac.jp}}
}
\date{}
\begin{document}

\maketitle
\begin{abstract}
   This paper aims to consider spherical quandles and give a one-to-one correspondence between $SU(2)$-representations of knot groups and colorings of knots with spherical quandles.
\end{abstract}


\section{Introduction}
In 1982, Joyce \cite{Joyce1982} and Matveev \cite{Matveev1982} defined an algebraic system called quandles and constructed the almost complete knot invariant called the knot quandle or the fundamental quandle of knots.

In knot theory, $SU(2)$-representations of knot groups are interesting subject of study since they give some invaluable knot invariants. For example, X.-S. Lin \cite{Lin1992} defined the Casson-Lin invariant as an analogue of the original Casson invariant \cite{AkbulutMcCarthy}. Surprisingly, Lin proved that the invariant is one half of the signature of knots. After that, Herald \cite{Herald1997} and Heusener-Kroll \cite{HeusenerKroll1998} extended the invariant independently.

There are two quandles which seems different but have same name. In 1994, Azcan-Fenn \cite{AzcanFenn1994} defined the spherical quandle and ambiguously suggested the connection between $SU(2)$-representations of knot groups and quandle homomorphisms from knot quandles to the spherical quandle on the 2-sphere. On the other hand, in 2018, Clark-Saito \cite{ClarkSaito} defined a family of quandles on conjugacy classes of $SU(2)$ and called them spherical quandles. They defined a knot invariant called a longitudinal mapping and calculate it in the case of $SU(2)$ using the quandle.

The aims of this paper is to consider two problems: to consider the difference of the two definition of spherical quandles and to give a concrete extension of Azcan-Fenn's suggestion.

This paper is organized as follows. In section \ref{section_preliminaries}, the basic notation and facts on quandles and $SU(2)$ are presented. In section \ref{section_spherical_quandles}, we recall the definitions of spherical quandles by Azcan-Fenn and by Clark-Saito, and show that the former is a kind of the latter. In section \ref{section_1to1correspondence}, we give a one-to-one correspondence between $SU(2)$-representations and spherical quandle-colorings extending Azcann-Fenn's suggestion. %The last section includes an explanation of the condition ``tace-free" through the quandle theory.

\section{Preliminaries}\label{section_preliminaries}
We recall definitions and facts using in this paper without proofs.

\subsection{Quandle}
We see the definition of a quandle and some basic facts. See Kamada \cite{Kamadabook} and Nosaka \cite{Nosaka2017book} for more details.
\begin{definition}[Joyce \cite{Joyce1982}, Matveev \cite{Matveev1982}]
\label{def_quandle}
A \textit{quandle} is a set $X$ with a binary operation $\triangleright:X\times X\to X$ satisfying the three conditions:

\noindent(Q1) $x\triangleright x = x$ for any $x\in X$.

\noindent(Q2) The map $S_y:X\to X$ defined by $x\mapsto x\triangleright y$ is bijective for any $y\in X$.

\noindent(Q3) $(x\triangleright y)\triangleright z = (x\triangleright z)\triangleright(y\triangleright z)$ for any $x,y,z\in X$.
\end{definition}

\begin{Example}
Suppose $X$ is a set and $x\triangleright y=x$ for any $x,y\in X$. Then $(X,\triangleright)$ is a quandle called a \textit{trivial quandle}.
\end{Example}

\begin{Example}Suppose $X=\mathbb{Z}/n\mathbb{Z}$ and $x\triangleright y=2y-x$ for any $x,y\in X$. Then $(X,\triangleright)$ is a quandle called a \textit{dihedral quandle} \cite{Takasaki1943}.
\end{Example}

\begin{Example}
\label{def_augmented_quandle}
Let $G$ be a group, $X$ a set on which $G$ acts from the right, and  $\kappa:X\to G$ a map satisfying the two conditions:
\begin{enumerate}
    \item $\kappa(x\cdot g)=g^{-1}\kappa(x)g$ for any $x\in X$ and $g\in G$.
    \item $x\cdot \kappa(x)=x$ for any $x\in X$.
\end{enumerate}
Suppose $x\triangleright y=x\cdot\kappa(y)$ for any $x,y\in X$. Then $(X,\triangleright)$ is a quandle called an \textit{augmented quandle} \cite{Joyce1982}. We denote this quandle by $(X,G,\kappa)$ or simply $X$. The quandle $X$ is said to be \textit{faithful} if the map $\kappa$ is injective.
\end{Example}

\begin{Example}[Eisermann \cite{Eisermann2014}]
\label{Eisermann augmented quandle}
Let $G$ be a Lie group and $\mathfrak{g}$ be the Lie algebra of $G$. The Lie group $G$ acts on $\mathfrak{g}$ via the adjoint right action: $X\cdot g:=g^{-1}Xg$ for any $X\in \mathfrak{g}$ and $g\in G$. Then $(\mathfrak{g},G,\operatorname{exp})$ is an augumented quandle , where $\operatorname{exp}:\mathfrak{g}\to G$ is the exponential map.
\end{Example}

\begin{Example}
\label{knot quandle_def}
Let $K$ be a tame knot in the 3-sphere, $\pi_K=\pi_1(S^3\setminus K)$ the knot group of $K$, and $H=\langle\mathfrak{m},\mathfrak{l\rangle}$ the subgroup of $\pi_K$ generated by a meridian $\mathfrak{m}\in\pi_K$ and the preferred longitude $\mathfrak{l}\in \pi_K$ of $K$. Suppose $X=H\backslash\pi_K$ and $Hx\triangleright Hy=H\mathfrak{m}xy^{-1}\mathfrak{m}y$ for $x,y\in\pi_K$. Then, the algebraic system $(X,\triangleright)$ is a quandle and called the \textit{knot quandle} of $K$ or the \textit{fundamental quandle} of $K$ \cite{Joyce1982,Matveev1982}. We denote this quandle by $Q_K$. 
\end{Example}
A subset $Y$ of quandle $X$ is said to be a \textit{subquandle} if the quandle operation of $X$ is closed in $Y$.

A quandle is said to be an \textit{involutory quandle} or a \textit{kei} if
$
(x\triangleright y)\triangleright y=x
$ for any $x,y$. A dihedral quandle is an involutory quandle for example.

A map $f:X\to Y$ between quandles is said to be a \textit{quandle homomorphism} if $f(x\triangleright y)=f(x)\triangleright f(y)$ for any $x,y\in X$. A quandle homomorphism is said to be a \textit{quandle isomorphism} if it is bijective.

\begin{Example}\label{example_trivial-hom}
Suppose $X$ be a quandle and $Y=\{*\}$ be a trivial quandle. Then, any map $f:X\to Y$ is a quandle homomorphism.
\end{Example}

\begin{Example}
By (Q2) and (Q3) of Definition \ref{def_quandle}, $S_y$ is an automorphism of a quandle $X$ for any $y\in X$. We call $S_y$ an \textit{inner automorphism} of $X$. The subgroup of $\operatorname{Aut}X$ generated by all inner automorphisms is called the inner automorphism group  and denoted by $\operatorname{Inn}X$.
\end{Example}

\begin{definition}\label{def_trivial-hom}
In this paper, a quandle homomorphism $f:X\to Y$ is said to be \textit{trivial} if the image of $f$ is a trivial subquandle of $Y$.
\end{definition}

\begin{rem}
Some literature defines a trivial homomorphism as a quandle homomorphism considered in Example \ref{example_trivial-hom}.
\end{rem}

\begin{rem}
The knot quandle, we defined in Example \ref{knot quandle_def}, is a complete knot invariant. See \cite{Joyce1982, Matveev1982,Kamadabook, Nosaka2017book} for more details.
\end{rem}

For a knot $K$, we call a quandle homomorphism from the knot quandle $Q_K$ to a quandle $X$ as an $X$-\textit{colorling} of $K$.

\subsection{Properties of SU(2)}\label{section_property_SU(2)}
Consider the 2-dimentional special unitary group
\[
SU(2)=
\left\{
\begin{pmatrix}
a && b\\
-\overline{b} && \overline{a}
\end{pmatrix}
\ :\ {}
a,b\in\mathbb{C},\ a\overline{a}+b\overline{b}=1
\right\}
\]
and its Lie algebra
\[
\mathfrak{su}(2)=
\left\{
\begin{pmatrix}
ix && z\\
-\overline{z} && -ix
\end{pmatrix}
\ :\ {}
x\in\mathbb{R},\ z\in\mathbb{C}
\right\},
\]
where $i=\sqrt{-1}$. The Lie group $SU(2)$ acts on itself via conjugation and on $\mathfrak{su}(2)$ via adjoint representation, $g\cdot X=gXg^{-1}$ for any $g\in SU(2)$, $X\in \mathfrak{su}(2)$. Let $\operatorname{exp}:\mathfrak{su}(2)\to SU(2)$ be the exponential map and $S^2(r)$ be the 2-sphere
\[
S^2(r)=
\left\{
\begin{pmatrix}
ix && z\\
-\overline{z} && -ix
\end{pmatrix}
\in\mathfrak{su}(2)
\ :\ {}
x^2+z\overline{z}=
r^2
\right\}
\]
for $r\in \mathbb{R}\setminus\{0\}$. We recall the following known facts.
\begin{prop}\label{prop_S2r_SU(2)-orbit}
For any $r\in\mathbb{R}\setminus\{0\}$, $S^2(r)$ is an $SU(2)$-orbit.
\end{prop}

\begin{prop}
\label{prop_tr_expS2r}
For any $r\in \mathbb{R}\setminus\{0\}$,
\[
\operatorname{exp}S^2(r)
=\{g\in SU(2)\ :\ {}\operatorname{tr}g=2\cos r\}.
\]
\end{prop}

\begin{prop}\label{prop_injectivity_exp}
Let $r$ be an element of $\mathbb{R}\setminus\pi\mathbb{Z}$, and
\[
D(r)=
\begin{pmatrix}
ir && 0\\
0 && -ir
\end{pmatrix}
\in S^2(r).
\] 
Then, both of the isotropy subgroup of $SU(2)$ with respect to $D(r)\in\mathfrak{su}(2)$ and the isotropy subgroup with respect to $\operatorname{exp}D(r)\in SU(2)$ are the subgroup consisting of the entire diagonal matrices of $SU(2)$.
\end{prop}

%$\operatorname{exp}$ is $SU(2)$-equivariant, i.e. $\operatorname{exp}(gXg^{-1})=g\operatorname{exp}(X)g^{-1}$ for any $g\in SU(2)$, $X\in \mathfrak{su}(2)$, and $S^2(t)$ is a , 

%It is well known that two elements of $SU(2)$ are conjugate in $SU(2)$ if and only if their trace values are equal. 
\section{Spherical quandles}\label{section_spherical_quandles}
In this section, we consider subquandles of quandles defined in Example \ref{Eisermann augmented quandle}. We see that they are spherical quandles defined by Azcan-Fenn\cite{AzcanFenn1994} and by Clark-Saito\cite{ClarkSaito}.

\subsection{Definition of spherical quandles}
\begin{definition}[Azcan-Fenn \cite{AzcanFenn1994}]\label{def_spherical_quandle}
Let $\langle-,-\rangle:\mathbb{R}^3\times\mathbb{R}^3\to\mathbb{R}$ be the Euclidean inner product and $S^2$ the 2-sphere
\[
\left\{
\bm{x}=(x_1,x_2,x_3)\in \mathbb{R}^3\mid \langle \bm{x},\bm{x}\rangle=x_1^2+x_2^2+x_3^2=1
\right\}.
\]
We define the binary operation $\triangleright:S^2\times S^2\to  S^2$ as {}
$
\bm{x}\triangleright \bm{y}=2{\langle \bm{x},\bm{y}\rangle}\bm{y}-\bm{x}
$ {}
for all $\bm{x},\bm{y}\in S^2$.
Then $(S^2,\triangleright)$ is an involutory quandle and called the spherical quandle $S^2_{\mathbb{R}}$.
\end{definition}

\begin{rem}
Azcan-Fenn defined spherical quandles on the $n$-sphere $S^n$ in the same way (see %{Azcan-Fenn}
\cite{AzcanFenn1994} or %Nosaka
\cite{Nosaka2017book} for more details). However, we do not deal with them since we would like to compare with Clark-Saito's qunadles defined on the $2$-sphere $S^2$.
\end{rem}

To describe the definition of Clark-Saito's spherical quandle, we define some symbols. Suppose $\mathcal{S}^2$ is the pure unitquaternions
\[
\{a\bm{i}+b\bm{j}+c\bm{k} \mid a,b,c\in\mathbb{R},\ a^2+b^2+c^2=1\}.
\]
Then each  unit quaternion can be represented in the form
\[
\bm{e}^{\theta\bm{u}}=\cos\theta+(\sin\theta)\bm{u}
\]
with some $\bm{u}\in \mathcal{S}^2$ and $\theta\in(0,\pi)$.

\begin{definition}[Clark-Saito \cite{ClarkSaito}]
\label{def_Clark-Saito_spherical_quandle}
For $0<\psi<\pi$ we denote by $S^2_{\psi}$ the quandle with an underlying set $\mathcal{S}^2$ and a product $\bm{u}*\bm{v}=\bm{u}\operatorname{Rot}_{\psi}(\bm{v})$. We call it a spherical quandle in the sense of Clark-Saito.
\end{definition}

\begin{rem}\label{rem_ClarkSaito_lem4.4}
We identify unit quatenions with elements of $SU(2)$. Then we are able to identify spherical quandles in the sense of Clark-Saito with the subquandle of conjugacy quandle of $SU(2)$. See \cite[Lemma 4.4]{ClarkSaito}.
\end{rem}

\subsection{Subquandles of quandles in Example \ref{Eisermann augmented quandle}}
We consider the quandle defined in Example \ref{Eisermann augmented quandle} for the case $G=SU(2)$ and $\mathfrak{g}=\mathfrak{su}(2)$, and its subquandles. Let $S^2(r)$ be the subset of $\mathfrak{su}(2)$ defined in section \ref{section_property_SU(2)}. We give a few lemmas which we use in the latter sections, and a presentation of $S^2_{\mathbb{R}}$ as an augmented quandle.

\begin{lemma}
The 2-sphere $S^2(r)$ is a subquandle of $(\mathfrak{su}(2),SU(2),\operatorname{exp})$.
\end{lemma}
\begin{proof}
It is easy to see that the action of $SU(2)$ is well-defined on $S^2(r)$ in light of Proposition \ref{prop_S2r_SU(2)-orbit}.
\end{proof}

\begin{lemma}\label{lem_faithfulness_spherical_qnd}
For any $r\in\mathbb{R}\setminus\pi\mathbb{Z}$, the augmented quandle $(S^2(r),SU(2),\operatorname{exp}|_{S^2(r)})$ is faithful.
\end{lemma}
\begin{proof}
By Proposition \ref{prop_injectivity_exp}, the restriction of $\operatorname{exp}$ on $S^2(r)$ is injective. 
\end{proof}

To use in the proof of Lemma \ref{lem_Inn_of_(S^2(r))}, we quote the statement of Theorem 3.1 of \cite{Nosaka2017} rewritten the symbols to ours:
\begin{theorem}[Nosaka \cite{Nosaka2017}]
\label{thm_statement_Thm3.1_NSK}
Assume that a group $G$ acts on a quandle $X$ from the right, and a map $\kappa: X\to G$ satisfies the following conditions:
    \begin{enumerate}
        \item $x \triangleright y = x \cdot \kappa(y) \in X$ for any $x, y \in X$.
        \item The image $\kappa(X) \subset G$ generates the group $G$, and the action of $G$ on $X$ is effective.
    \end{enumerate}
    Then, there is an isomorphism between  $\operatorname{Inn}(X)$ and $G$, and the action of $G$ on $X$ coinsides with the natural action of $\operatorname{Inn}X$.
\end{theorem}

\begin{lemma}\label{lem_Inn_of_(S^2(r))}
For any $r\in\mathbb{R}\setminus\pi\mathbb{Z}$, the inner autmorphism group of the augmented quandle $S^2(r)$ is isomorphic to $SO(3)=SU(2)/\{\pm E\}$.
\end{lemma}
\begin{proof}
The inner automorphism group of $(S^2(r),SU(2),\operatorname{exp}|_{S^2(r)})$ is isomorphic to a subgroup of $SO(3)=SU(2)/\{\pm E\}$ since the quandle operation of $S^2(r)$ is induced by the adjoint action of $SU(2)$ on $\mathfrak{su}(2)$.

Thus it is sufficient to see that the action of $SO(3)$ on $S^2(r)$ satisfies the conditions of Theorem \ref{thm_statement_Thm3.1_NSK}.

By the second condition of Example \ref{def_augmented_quandle}, the first condition of Theorem \ref{thm_statement_Thm3.1_NSK} is satisfied. By Cartan–Dieudonn\'{e} theorem (see \cite[Chapter I Section 10]{Cartan1966}), the former of the second condition of Theorem \ref{thm_statement_Thm3.1_NSK} is satisfied. By Proposition \ref{prop_injectivity_exp}, the action of $SO(3)=SU(2)/\{\pm E\}$ on the 2-sphere $S^2(r)$ is effective since the $SO(3)$-action on $\mathbb{R}^3$ via linear transformation is effective.
\end{proof}

\begin{rem}
The method of the proof of Lemma \ref{lem_Inn_of_(S^2(r))} is same as the proof of \cite[Lemma 4.9]{Nosaka2017}.
\end{rem}

\begin{rem}\label{rem_ClarkSaito-qunadle}
By Lemma \ref{lem_faithfulness_spherical_qnd} and Remark \ref{rem_ClarkSaito_lem4.4}, the spherical quandle in the sense of Clark-Saito $S^2_{2\pi-2r}$ is isomorphic to our augmented quandle $S^2(r)$ for $r\in(0,\pi)$.
\end{rem}

\subsection{Presentation of $S^2_{\mathbb{R}}$ as an augmented quandle}
We give a presentation of $S^2_{\mathbb{R}}$ as an augmented quandle and prove the spherical quandle $S^2_{\mathbb{R}}$ is isomorphic to the spherical quandle $S^2_{\pi}$. 
\begin{theorem}\label{prop_augmented_qnd_presentation_spherical_qnd}
The map $h:
S_{\mathbb{R}}^2\to (S^2(\pi /2),SU(2),\operatorname{exp}|_{S^2(\pi/2)})$
defined by
\[
h(x_1,x_2,x_3)=
\frac{\pi}{2}
\begin{pmatrix}
x_1 i && x_2+x_3 i\\
-x_2+x_3 i && -x_1 i
\end{pmatrix}
\]
is a quandle isomorphism.
\end{theorem}
\begin{proof}
It is sufficient to see that $h$ is a quandle homomorphism, because $h$ is obviously bijective. For any $\bm{y}=(y_1,y_2,y_3)\in S^2_{\mathbb{R}}$, the eigenvalues of $h(\bm{y})$ are $\pm \frac{\pi}{2}i$. Therefore, $h(\bm{y})$ is diagonalizable and
\[
\operatorname{exp}h(\bm{y})
=
\begin{pmatrix}
y_1 i && y_2+y_3 i\\
-y_2+y_3 i && -y_1 i
\end{pmatrix}
=\frac{2}{\pi}h(\bm{y})
.
\]
Hence, for any $\bm{x}=(x_1,x_2,x_3),\bm{y}=(y_1,y_2,y_3)\in S^2_{\mathbb{R}}$,
\begin{eqnarray*}
h(\bm{x}\triangleright\bm{y})
=
\frac{\pi}{2}
\begin{pmatrix}
\alpha i& \beta+\gamma i\\
-\beta+\gamma i & -\alpha i
\end{pmatrix}
=h(\bm{x})\triangleright h(\bm{y}),
\end{eqnarray*}
where $\alpha=2\langle\bm{x},\bm{y}\rangle y_1-x_1$, 
$\beta=2\langle\bm{x},\bm{y}\rangle y_2-x_2$ and $\gamma=2\langle\bm{x},\bm{y}\rangle y_3-x_3$.
\end{proof}

\begin{rem}
In light of Theorem \ref{prop_augmented_qnd_presentation_spherical_qnd}, \cite[Lemma 4.9]{Nosaka2017} is obtained from Lemma \ref{lem_Inn_of_(S^2(r))} considering the case $r=\frac{\pi}{2}$.
\end{rem}

%\label{rem_ClarkSaito-quandle-spherical_quandle}
By Remark \ref{rem_ClarkSaito-qunadle} and Proposition \ref{prop_augmented_qnd_presentation_spherical_qnd}, the quandle $S^2_{\pi}$ defined in Definition \ref{def_Clark-Saito_spherical_quandle} is isomorphic to the quandle $S^2_{\mathbb{R}}$ defined in Definition \ref{def_spherical_quandle}.

\section{One-to-one correspondence}\label{section_1to1correspondence}
In this section, we give a one-to-one correspondence between $SU(2)$-representations of knot groups and colorings of knots with spherical quandles. It extends the suggetion by Azcan-Fenn \cite{AzcanFenn1994}. We use Nosaka's work to give the correspondence.

Let $K$ be a tame knot in the 3-sphere $S^3$, $\pi_K=\pi_1(S^3\setminus K)$ the knot group of $K$ and $Q_K$ the knot quandle. For any augmented quandle $(X,G,\kappa)$, we define a set
\[
R(K,G)=\{
f\in\operatorname{Hom}(\pi_K,G)
\ :\ {}
{}^{\exists}x\in X,\ f(\mathfrak{m})=\kappa(x)
\}.
\]
\begin{theorem}[Nosaka \cite{Nosaka2015}]\label{thm_Nosaka_correspondence}
Let $(X,G,\kappa)$ be a faithful augmented quandle. Then, there is a bijection
\[
\Psi:\operatorname{Hom}(Q_K,X)
\xrightarrow{\sim}
R(K,G).
\]
\end{theorem}
The bijection $\Psi$ is given as follows. Suppose $D$ is a diagram of $K$. It is known that both $\pi_K$ and $Q_K$ are generated by the elements corresponding to the arcs of $D$ (see \cite{Kamadabook}). For any quandle homomorphism $f:Q_K\to X$, $\Psi f:\pi_K\to G$ is a group homomorphism satisfying {}
$
\Psi f(\alpha)=\kappa\circ f(\alpha)
$ {}
for any $\alpha$ of $Q_K$ corresponding to an arc of $D$.

We give a few facts about the bijection $\Psi$.
\begin{lemma}\label{lem_Nosaka-1to1_correspondence}
\begin{enumerate}
    \item The action of $G$ on $X$ induces the action of $G$ on $\operatorname{Hom}(Q_K,X)$. Quandle homomorphisms $f,g:Q_K\to X$ are in the same $G$-orbit if and only if $\Psi f$ and $\Psi g$ are conjugate.
    \item A quandle homomorphism $f:Q_K\to X$ is trivial in the sense of Definition \ref{def_trivial-hom} if and only if $\Psi f$ is an abelian representation.
\end{enumerate}
\end{lemma}

\begin{theorem}\label{main_theorem}
For any $r\in(0,\pi)$, there is a bijection
\begin{eqnarray*}
\Phi_{K, r}:
\operatorname{Hom}(Q_K,S^2(r))
\xrightarrow{\sim}
\left\{
\rho\in\operatorname{Hom}(\pi_K,SU(2))
:
\operatorname{tr}\rho(\mathfrak{m})=2\cos r
\right\}.
\end{eqnarray*}
\end{theorem}

\begin{proof}
By Lemma \ref{lem_faithfulness_spherical_qnd}, $(S^2(r),SU(2),\operatorname{exp})$ is faithful. Thus, we have the bijection $\Phi_{K, r}$ in light of Proposition \ref{prop_tr_expS2r} and Theorem \ref{thm_Nosaka_correspondence}. By Lemma \ref{lem_Nosaka-1to1_correspondence}, we have the following properties.
%The condition 1 follows by 
%The necessity of condition 2 follows by Lemma \ref{lem_Nosaka-1to1_correspondence}. Suppose abelian representation $\rho:\pi_K\to SU(2)$ satisfies $\operatorname{tr}\rho(\mathfrak{m})=2\cos r$. Then, $\rho(\mathfrak{m})$ is conjugate to $\operatorname{exp}D(r)$ (see Proposition \ref{prop_injectivity_exp} for the definiton of $D(r)$), i.e. there exists a $g\in SU(2)$ satisfing $\rho(\mathfrak{m})=g^{-1}\left(\operatorname{exp}D(r)\right)g$. By Proposition \ref{prop_injectivity_exp}, the image of $\Phi_{K, r}^{-1}\rho$ is $\{g^{-1}D(r)g\}$.
\end{proof}
By Lemma \ref{lem_Nosaka-1to1_correspondence}, we have the following properties.
\begin{enumerate}
    \item The action of $SU(2)$ on $S^2(r)$ induces the action of $SU(2)$ on $\operatorname{Hom}(Q_K,S^2(r))$. Then, quandle homomorphisms $f,g:Q_K\to (S^2(r),SU(2),\operatorname{exp})$ are in the same $SU(2)$-orbit if and only if $\Phi_{K, r}f$ and $\Phi_{K, r} g$ are conjugate.
    \item A quandle homomorphism $f:Q_K\to (S^2(r),SU(2),\operatorname{exp})$ is trivial in the sense of Definition \ref{def_trivial-hom} if and only if $\Phi_{K, r}f$ is abelian.
\end{enumerate}

\begin{table}[t]
  \begin{tabular}{ccc}\hline
  \begin{tabular}{c}
   Fixed-trace\\ $SU(2)$-rep. of $\pi_K$
  \end{tabular}
    & Interpretation by $\Phi_{K, r}$ & Reference \\\hline\hline
    rep. satisfying $\operatorname{tr}\rho(\mathfrak{m})=2\cos r$ & $S^2(r)$-colorings &
Thm. \ref{main_theorem}    
    \\\hline
    abelian rep. & 
    \begin{tabular}{c}
trivial colorings in the \\sense
of Definition \ref{def_trivial-hom}
    \end{tabular}
    & Thm. \ref{main_theorem}
    \\\hline
    conjugate & in the same $\operatorname{Inn}S^2(r)$-orbit & Rem. \ref{rem_SO(3)-orbit}\\\hline
  \end{tabular}
  \caption{Summary of the discussion in section \ref{section_1to1correspondence}}
  \label{table_summary_general}
\end{table}

\begin{rem}
Theorem \ref{main_theorem} extends the argument \cite[Lemma 4.1]{AzcanFenn1994} in light of Proposition \ref{prop_augmented_qnd_presentation_spherical_qnd}.
\end{rem}

\begin{rem}\label{rem_SO(3)-orbit}
By Theorem \ref{thm_statement_Thm3.1_NSK} and Lemma \ref{lem_Inn_of_(S^2(r))}, $SU(2)$-orbits of $\operatorname{Hom}(Q_K,S^2(r))$ coincide with $SO(3)\cong\operatorname{Inn}S^2(r)$-orbits of it.
\end{rem}

We restate the theorem in the case of $S^2_{\mathbb{R}}$.
\begin{cor}\label{cor_1to1correspondence_spherical_quandles}
There is a bijection
\begin{eqnarray*}
\Phi_{K, S^2}:
\operatorname{Hom}(Q_K,S^2_{\mathbb{R}})
\xrightarrow{\sim}
\left\{
\rho\in\operatorname{Hom}(\pi_K,SU(2))
:
\operatorname{tr}\rho(\mathfrak{m})=0
\right\}.
\end{eqnarray*}
\end{cor}

The class of $SU(2)$-representations in Colloray \ref{cor_1to1correspondence_spherical_quandles} corresponds to the class appeared in \cite{Lin1992}. It may be possible to interpret the Casson-Lin invariant by quandle theory.

\begin{table}[htb]
  \begin{tabular}{ccc}\hline
  \begin{tabular}{c}
   Trace-zero\\ $SU(2)$-rep. of $\pi_K$
  \end{tabular}
    & Interpretation by $\Phi_{K, S^2}$ & Reference \\\hline\hline
    trace-zero rep. & $S^2_{\mathbb{R}}$-colorings &
Prop. \ref{prop_augmented_qnd_presentation_spherical_qnd}
and Thm. \ref{main_theorem}    
    \\\hline
    abelian rep. & 
    \begin{tabular}{c}
trivial colorings in the \\sense
of Def. \ref{def_trivial-hom}
    \end{tabular}
    & Prop. \ref{prop_augmented_qnd_presentation_spherical_qnd}
and Thm. \ref{main_theorem}
    \\\hline
    conjugate & in the same $\operatorname{Inn}S^2_{\mathbb{R}}$-orbit & Prop. \ref{prop_augmented_qnd_presentation_spherical_qnd} and Rem. \ref{rem_SO(3)-orbit}\\\hline
  \end{tabular}
  \caption{Summary of the discussion in section \ref{section_1to1correspondence} in the case $S^2_{\mathbb{R}}$}
  \label{table_summary}
\end{table}
\section*{Acknowledgement}
The author is grateful to Professor Hiroyuki Ochiai, Kyushu University, for many valuable comments and discussions. He also thanks Professor Takefumi Nosaka, Tokyo Institute of Technology, for referring him to the paper \cite{ClarkSaito}.

\textit{Department of Mathematics, Kyushu University, 744 Motooka, Nishi-ku, Fukuoka 819–0395, Japan}
\end{document}